\documentclass[11pt]{article}

\usepackage{cite}
\usepackage{amssymb}
\usepackage{amsthm}
\usepackage{amsmath}
\usepackage{amsfonts}
\usepackage{txfonts}
\usepackage{hyperref}
\usepackage{color}

\usepackage[T1]{fontenc}
\usepackage[latin2]{inputenc}
\title{Discrete Cucker-Smale's flocking model with a weakly singular weight}
\author{Jan Peszek\footnote{{\it Email address:} {\tt j.peszek@mimuw.edu.pl}}\\
{\it\footnotesize Institute of Applied Mathematics and Mechanics,}\\
{\it\footnotesize University of Warsaw,}\\
{\it\footnotesize ul. Banacha 2, 02-097 Warsaw, Poland}}

\date{\today}

\renewcommand{\it}{\sl}

plus 6pt minus 12pt
\newcommand{\barint}{
         \rule[.036in]{.12in}{.009in}\kern-.16in
          \displaystyle\int  }
\def\r{{\mathbb{R}}}

\begin{document}

\newtheorem{theo}{\bf Theorem}[section]
\newtheorem{coro}{\bf Corollary}[section]
\newtheorem{lem}{\bf Lemma}[section]
\newtheorem{rem}{\bf Remark}[section]
\newtheorem{defi}{\bf Definition}[section]
\newtheorem{ex}{\bf Example}[section]
\newtheorem{fact}{\bf Fact}[section]
\newtheorem{prop}{\bf Proposition}[section]
\newtheorem{prob}{\bf Problem}[section]

\makeatletter \@addtoreset{equation}{section}
\renewcommand{\theequation}{\thesection.\arabic{equation}}
\makeatother

\newcommand{\ds}{\displaystyle}
\newcommand{\ts}{\textstyle}
\newcommand{\ol}{\overline}
\newcommand{\wt}{\widetilde}
\newcommand{\ck}{{\cal K}}
\newcommand{\ve}{\varepsilon}
\newcommand{\vp}{\varphi}
\newcommand{\pa}{\partial}
\newcommand{\rp}{\mathbb{R}_+}
\newcommand{\hh}{\tilde{h}}
\newcommand{\HH}{\tilde{H}}
\newcommand{\cp}{{\rm cap}^+_M}
\newcommand{\hes}{\nabla^{(2)}}
\newcommand{\nn}{{\cal N}}
\newcommand{\dix}{\nabla_x\cdot}
\newcommand{\dv}{{\rm div}_v}
\newcommand{\di}{{\rm div}}
\newcommand{\pxi}{\partial_{x_i}}
\newcommand{\pmi}{\partial_{m_i}}
\newcommand{\tor}{\mathbb{T}}
\newcommand{\pot}{\mathcal{v}}

\maketitle
\begin{abstract}
For the discrete Cucker-Smale's flocking model with a singular communication weight $\psi(s)=s^{-\alpha}$, with $0<\alpha<\frac{1}{2}$, we prove that the velocity component of certain type of weak solutions is absolutly continuous. This result enables us to obtain existence and uniqeness of global solutions.
\end{abstract}

\section{Introduction}\label{s1}
Mathematical description of a collective self-driven motion of self-proppeled agents that have a tendency to allign their velocities appears in many applications including modelling of flocks of birds or schools of fish. Also some seemingly unrelated phenomena such as emergence of common languages in primitive societies, distribution of goods, optimal control over sensor networks or reaching a consensus in decision making models can be described as a collective motion of agents with a tendency to flock. In \cite{cuc1} from 2007, Cucker and Smale introduced a model for the flocking of birds, that to some extent, was based on the paper by Vicsek (\cite{vic}) from 1995 and is the subject of our interest.
In \cite{jpe} we proved existence of {\it piecewise weak solutions}, i.e. a special class of solutions to the discrete Cucker-Smale's flocking model (C-S) with a singular communication weight. The most significant property of piecewise weak solutions was that the trajectories could stick together in the sense, that two different trajectories begin to coincide at some time $t_0$ and they coincide indefinitely or at least in some time interval of a positive length. In any neighborhood of such time of sticking, a significant loss of regularity occurs and thus, the solution has to be somehow 'cut off' of such times. Here, our goal is to prove that such solutions preserve sufficient regularity also at the times of sticking, which enables us to prove existence and uniqueness in a better class of regularity.

Let us briefly introduce the model. We assume that there exist $N$ particles in $\r^d$ and that their position and velocity, denoted for $i$-th particle by $(x_i,v_i)$ are governed by the following system of ODE's
\begin{eqnarray}\label{cs}
\left\{
\begin{array}{ccl}
\frac{d}{dt}x_i &=& v_i,\\
\frac{d}{dt}v_i &=& \sum_{j=1}^N(v_j-v_i)\psi(|x_j-x_i|),
\end{array}
\right.
\end{eqnarray}
where $\psi:[0,\infty)\to[0,\infty)$ is the aforementioned {\it communication weight}. In general, the communication weight is a given nonnegative, nonincreasing function that is commonly interpreted as the perception of the particles. The most classic example of a smooth communication weight is
\begin{eqnarray}\label{cucu}
\psi_{cs}(s)=\frac{K}{(1+s^2)^\frac{\beta}{2}}, \ \ \ \beta\geq 0,\ \ \  K>0
\end{eqnarray}
introduced by Cucker and Smale in \cite{cuc1}. For such weight (or more generally -- for all bounded and Lipschitz continuous weights) C-S model was extensivly studied both from physical (see e.g.. \cite{deg1, deg2, deg3, park, top}) and mathematical (see e.g.. \cite{aha2, can, cuc2, cuc3, cuc4, dua1, hakalaru, hajekaka, hahaki, hajuslem, halele, halaruslem, haslem, mo, shen, car, car3}) point of view.  
As in case of many other particle system's governed by a Newtonian law, the microscopic description can be replaced by the mezoscopic one, in which case, C-S model is associated with the following Vlasov-type equation:
\begin{eqnarray}
\partial_tf+v\cdot\nabla f+{\rm div}_v(F(f)f)=0,\ \ x\in\r^d,\ v\in\r^d,\label{cscont}\\
F(f)(x,v,t):=\int_{\r^{2d}}\psi(|y-x|)(w-v)f(y,w,t)dwdy\label{force},
\end{eqnarray}
where $f=f(x,v,t)$ is the density of particles that at the time $t$ have position $x$ and velocity $v$. Passage from (\ref{cs}) to (\ref{cscont}) can be done via a mean-field limit and in case of regular communication weight can be found in \cite{haliu} or \cite{hatad}. For a more general overview of the passage from particle systems to continuous equations similar to (\ref{cscont}) in aggregation and swarming models via a mean-field limit we refer to \cite{rec}.

Our main interest is the C-S model with the weight
\begin{eqnarray}\label{psi}
\psi(s)=
\left\{
\begin{array}{ccc}
s^{-\alpha}&{\rm for}& s>0,\\
\infty&{\rm for}& s=0,
\end{array}
\ \ \ \ \ \ \alpha>0.\right.
\end{eqnarray}
Until now, two articles where published, that dealt with C-S model with a singular weight in two separate cases of $\alpha>1$ and $\alpha\in(0,1)$. This distinction plays an interesting role in how the problem is approached from different perpectives, ane applying unintegrability of $\psi$ for $\alpha>1$ and one -- integrability of $\psi$ for $\alpha\in(0,1)$. The first paper -- \cite{ahn1} from 2013 -- dealt with the case of $\alpha\in(1,\infty)$ and was based on a simple observation that in one dimensional case ($d=1$), with two particles, due to $\psi$ being unintegrable near $0$, no collision can occur in any finite time. In the paper \cite{ahn1}, the authors established a set of initial conditions for which the particles could not collide with $d\geq 1$ and for more than two particles. Now since the singularity of $\psi$ occurs only at $s=0$, the standard approach appropriate for the smooth communication weight is aplicable. The second article that approached the problem of singular weight was \cite{jpe} in which  the author considered $\alpha\in(0,1)$ and obtained existence of the aforementioned piecewise weak solutions making use of the fact that in such case $\psi$ is integrable in a neighborhood of $0$. We will state the results of \cite{jpe} precisely in the next section.\\
Recently a new development in this direction was presented in \cite{carchoha}, where the authors prove local well-posedness of continuous C-S model with singular weight and nonlinear dependance on velocity in the force term (\ref{force}).
\subsection{Main goal: the case of weakly singular weight}
In this paper we aim to improve the results of \cite{jpe} when $\alpha\in(0,\frac{1}{2})$. We will show that in such case for any initial data, the piecewise weak solution is absolutely continuous, unique and satisfies (\ref{cs}) in a $W^{2,1}$ weak sense (and in particular a.e.), which is significantly better than, what we were able to prove in \cite{jpe}. In the proof we use almost all results from \cite{jpe}. The improvement comes mostly from an inequality (Lemma \ref{kp}) orginating from \cite{ka1}, that enables us to show a better regularity of the solutions. This inequality is also the reason why we restrict the set of admissible $\alpha$ to $(0,\frac{1}{2})$ as for $\alpha \in (\frac{1}{2},1)$ it fails to be true and for $\alpha=\frac{1}{2}$ it just does not suffice.\\
Uniqueness of solutions to C-S model with singular weight along with the proven in \cite{jpe} possibility of sticking of the trajectories leads to an interesting phenomenon. Namely, we begin with what one could consider a standard ODE of the form (\ref{cs}) with regular weight. In particular solutions are unique and as expected the backwards in time problem for this ODE is also well possed. However as soon as we change $\psi$ to be of the form (\ref{psi}) with $\alpha\in(0,1)$, we end up with an ODE, with uniqueness of the solutions but due to sticking of the trajecories, without a well possed backwards in time problem. We aim to further study this phenomenon in the future.

\section{Preliminaries and notation}
In section \ref{reg} we prove the absolute continuity of the solutions to the discrete C-S model with a communication weight given by (\ref{psi}) and $\alpha\in(0,\frac{1}{2})$. The reasoning is based on our results from \cite{jpe} and the inequality from \cite{ka1}. In section \ref{uni} we prove the uniqueness.\\
Hereinafter $x=(x_1,...,x_N)\in\r^{Nd}$, where $x_i=(x_{i,1},...,x_{i,d})$ denotes the position of the particles, $v=\dot{x}$ is their velocity, while $N$ and $d$ are respectively the number of the particles and the dimension of the space. Moreover by $B_i(t)$ we will denote the set of all indexes $j$, such that up to the time $t$, the trajectory of $x_j$ does not coincide with the trajectory of $x_i$. Assuming that the trajectories, once coinciding cannot separate, we may define it as
\begin{eqnarray}\label{bi}
B_i(t):=\{k=1,...,N: x_k(t)\neq x_i(t)\ {\rm or}\ v_k(t)\neq v_i(t)\},
\end{eqnarray}
since any two particles with sufficiently smooth trajectories have the same position and velocity at the time $t$, if and only if they move on the same trajectory. Further, by $W^{k,p}(\Omega)$ we denote the Sobolev space of the functions with up to $k$-th weak derivative belonging to the space $L^p(\Omega)$.\\
We will say that particles $x_i$ and $x_j$ {\it collide} at the time $t$ if and only if $x_i(t)=x_j(t)$ and $v_i(t)\neq v_j(t)$ and we will say that they {\it stick together} at the time $t$ if and only if $x_i(t)=x_j(t)$ and $v_i(t)=v_j(t)$. Throughout the paper $C$ denotes a generic positive constant that may change from line to line even in the same inequality.\\
Here we present a summary from paper \cite{jpe}. We precisely state the definition of piecewise weak solutions and in Proposition \ref{sum} we state most of the results obtained throughout \cite{jpe}.
\begin{defi}\label{sol}
Let $0=T_0\leq T_1\leq ... \leq T_{N_s}$, be the set of all times of sticking and $T_{N_s+1}:=T$ be given positive number. For $n\in \{0,...,K\}$, on each interval $[T_n,T_{n+1}]$, we consider the problem
\begin{eqnarray}\label{csref}
\left\{
\begin{array}{lll}
\frac{dx_i}{dt}=v_i,\\
\frac{dv_i}{dt}=\frac{1}{N}\sum_{k\in B_i(T_n)}(v_k-v_i)\psi_n(|x_k-x_i|),\\
x_i\equiv x_j \ \ \ {\rm if}\ \ \ j\notin B_i(T_n)
\end{array}
\right.
\end{eqnarray}
for $t\in[T_n,T_{n+1}]$, with initial data $x(T_n),v(T_n)$.\\
We say that $x$ solves (\ref{csref}) on the time interval $[0,T]$, with weight given by (\ref{psi}) and arbitrary initial data $x(0)=x_0, v(0)=v_0$ if and only if for all $n=0,...,K$ and arbitrarly small $\epsilon>0$, the function $x\in (C^1([0,T]))^{Nd}$ is a weak in $(W^{2,1}([T_{n},T_{n+1}-\epsilon]))^{Nd}$ solution of (\ref{cs}).
\end{defi}

\begin{rem}\rm
In Definition \ref{sol} the purpose of redefining of the system (\ref{csref}) at each time of sticking $T_n$ (by including the set $B_i(T_n)$) is to ensure that once stuck together trajectories cannot separate. In this paper we prove existence and uniqueness of $W^{2,1}$ solutions to (\ref{cs}) and for such solutions trajectories cannot separate anyway and thus there is no need to redefine the system as in Definition \ref{sol}. However we will also prove that the solutions in the sense of Definition \ref{sol} are also unique (whether they belong to $W^{2,1}$ or not -- see Theorem \ref{main2}) and in this case sets $B_i(T_n)$ are crucial.
\end{rem}

\begin{prop}[Summary]\label{sum}
Let $\alpha\in(0,1)$. For all initial data $x_0,v_0$, there exists at least one solution of Cucker-Smale's flocking model with a singular communication weight given by (\ref{psi}). This solution exists in the sense of Definition \ref{sol}. Moreover the following properties hold:
\begin{enumerate}
\item For all $\epsilon>0$, the function $v$ is absolutely continuous on each time interval $[T_n,T_{n+1}-\epsilon]$.
\item The set of times of collision is at most countable, while the set of times of sticking has at most $N$ elements. Moreover if there exists a point of density of the times of collision (which we do not know whether exists or not) then this point itself is a time of sticking. Thus there are at most $N$ points of density of the times of collision.
\item Both $x$ and $v$ are uniformly bounded i.e. there exists an $N$ independent constant $C_1$ such that
\begin{eqnarray*}
\max_{i=1,...,N}\sup_{t\in[0,T]}|x_i(t)|\leq TC_1,\\
\max_{i=1,...,N}\sup_{t\in[0,T]}|v_i(t)|\leq C_1.
\end{eqnarray*}
\end{enumerate}
\end{prop}
\begin{proof}
The proof can be found in \cite{jpe}. Existence of solutions in the sense of Definition \ref{sol} is the subject of Theorem 2.1, point 1. is proved in Proposition 2.5, point 2. in Theorem 2.1 and point 3. in Corollary 2.2.
\end{proof}
\begin{rem}\rm
There is one seemingly significant difference between statement of the definition of the solution from \cite{jpe} and Definition \ref{sol}. Namely, in \cite{jpe} weak existence is stated to hold between times of collision instead of times of sticking (and there are significantly more times of collisions than times of sticking, which would suggest that the definition from \cite{jpe} was weaker). However it can be improved in a straightforward manner, since by Proposition 2.5 from \cite{jpe} the solution is absolutely continuous except only for left sided neighborhoods of times of sticking, which means that in fact we have weak existence as stated in Definition \ref{sol}.
\end{rem}

\section{Main result}
The main goal of this paper is presented in the form of the following theorems. The first theorem states that for $0<\alpha<\frac{1}{2}$ there exists a unique solution to (\ref{cs}) and it is reasonably regular.
\begin{theo}\label{main}
Let $\alpha\in(0,\frac{1}{2})$ be given. Then for all $T>0$ and arbitrary initial data there exists a unique $x\in W^{2,1}([0,T])\subset C^1([0,T])$ that solves (\ref{cs}) with communication weight given by (\ref{psi}) weakly in $W^{2,1}([0,T])$.
\end{theo}
The second theorem states that for $0<\alpha<1$ the piecewise weak solutions considered in \cite{jpe} are unique, even though they lack the $W^{2,1}([0,T])$ regularity.
\begin{theo}\label{main2}
Let $\alpha\in(0,1)$ be given. Then the solution in the sense of Definition \ref{sol}, which existence is ensured by Proposition \ref{sum} is unique.
\end{theo}

In case of $\alpha\in(0,\frac{1}{2})$, by Proposition \ref{sum} it suffices to prove uniqueness and that $v\in W^{1,1}([0,T])$ (i.e. that $v$ is absolutely continuous). In case of $\alpha\in(0,1)$ we only need to prove uniqueness. We do it in the subsequent sections.

\subsection{Absolute continuity of the velocity}\label{reg}
In this section we prove the absolute continuity of $v$. First let us state it in an explicit way.
\begin{prop}\label{abs}
With the assumptions of Theorem \ref{main}, there exists a constant $M$ depending only on the initial data, such that
\begin{eqnarray*}
\frac{1}{N}\sum_{i=1}^N\int_0^T|\dot{v}_i(t)|dt\leq M.
\end{eqnarray*}
Thus $v$ belong to the space $W^{1,1}([0,T])$ and is absolutely continuous.
\end{prop}
To prove the above Proposition we require the two presented below technical lemmas and the interpolation inequality from Appendix \ref{app}. The proofs of Lemmas \ref{lem1} and \ref{lem3} can be found at the end of the section.
\begin{lem}\label{lem1}
With the assumptions of Theorem \ref{main}, the function
\begin{eqnarray*}
R(t):=\sum_{i,j=1}^N|v_i(t)-v_j(t)|^2\psi(|x_i(t)-x_j(t)|)
\end{eqnarray*}
is integrable and
\begin{eqnarray*}
\int_0^TR(t)dt\leq N^2C_1^2,
\end{eqnarray*}
where $C_1$ is the constant from Proposition \ref{sum}.3.
\end{lem}

\begin{lem}\label{lem3}
Let the assumptions of Theorem \ref{main} be satisfied. Suppose further that there occurs no sticking in the time interval $[s_1,s_2]$. Then for all $i,j=1,...,N$ and all $\theta\in(0,1)$, we have
\begin{eqnarray*}
\int_{s_1}^{s_2}|x_j-x_i|^{-\theta}dt<\infty.
\end{eqnarray*}
\end{lem}

Now we proceed with the proof of Proposition \ref{abs}

\begin{proof}[Proof of Proposition \ref{abs}]
Let $T_k$ and $N_s\leq N$ be like in Definition \ref{sol}. Then, by Proposition \ref{sum}.1, velocity $v$ is absolutly continuous on each interval $[T_k,T_{k+1}-\epsilon]$ for arbitrarly small $\epsilon>0$. Therefore, given $k=0,...,N_s$ by $(\ref{cs})_2$, we have
\begin{eqnarray*}
\frac{1}{N}\sum_{i=1}^N\int_{T_k}^{T_{k+1}-\epsilon}|\dot{v}_i(t)|dt&=& \frac{1}{N}\sum_{i=1}^N\int_{T_k}^{T_{k+1}-\epsilon}\left|\frac{1}{N}\sum_{j=1}^N(v_j-v_i)\psi(|x_j-x_i|)\right|dt\\
&\leq& \frac{1}{N^2}\sum_{i,j=1}^N\int_{T_k}^{T_{k+1}-\epsilon}|v_j-v_i|\psi(|x_j-x_i|)dt.
\end{eqnarray*}
Let us denote
\begin{eqnarray*}
{\mathcal L}_{ijk}^\epsilon:=\int_{T_k}^{T_{k+1}-\epsilon}|v_j-v_i|\psi(|x_j-x_i|)dt.
\end{eqnarray*}
Then, we have
\begin{eqnarray*}
{\mathcal L}_{ijk}^\epsilon = \int_{T_k}^{T_{k+1}-\epsilon}|v_j-v_i|^{2\delta}(\psi(|x_j-x_i|))^\delta\cdot |v_j-v_i|^{1-2\delta}(\psi(|x_j-x_i|))^{1-\delta}dt,
\end{eqnarray*}
where $0<\delta<<1$ is some very small number. We then apply Young's inequality with $\eta>0$ and exponent $q=\frac{2}{1-2\delta}\in(1,\infty)$ (then it's conjugate $q^{'}=\frac{2}{1+2\delta}$) to get
\begin{eqnarray*}
{\mathcal L}_{ijk}^\epsilon &\leq& C(\eta)\int_{T_k}^{T_{k+1}-\epsilon} |v_j-v_i|^\frac{4\delta}{1+2\delta}(\psi(|x_j-x_i|))^\frac{2\delta}{1+2\delta}dt\\
&+&\eta C\int_{T_k}^{T_{k+1}-\epsilon}|v_j-v_i|^2(\psi(|x_j-x_i|))^\frac{2-2\delta}{1-2\delta}dt =: I_{ijk}^\epsilon+II_{ijk}^\epsilon.
\end{eqnarray*}
By H\" older's inequality with $q=\frac{1+2\delta}{2\delta}, q^{'}=1+2\delta$, we have
\begin{eqnarray}\label{jeden}
I_{ijk}^\epsilon\leq C(\eta)\left(\int_{T_k}^{T_{k+1}-\epsilon}|v_j-v_i|^2\psi(|x_j-x_i|)dt\right)^\frac{2\delta}{1+2\delta} \cdot(T_{k+1}-\epsilon-T_{k})^\frac{1}{1+2\delta}.
\end{eqnarray}
To deal with the estimation of $II_{ijk}^\epsilon$ we use Lemma \ref{lemkp}. First let us check whether, the assumptions are satisfied. However we will check if the assumptions are satisfied on $[T_k,T_{k+1}]$ instead of $[T_k,T_{k+1}-\epsilon]$ since we need estimates to be uniform with respect to $\epsilon$ anyway. We take $f=x_j-x_i$, which by Proposition \ref{sum} is a vector valued $C^1([T_{k},T_{k+1}])\cap W^{2,1}_{loc}((T_{k},T_{k+1}))$ function that is equal to $0$ in at most countable subset of $[T_{k},T_{k+1}]$. Moreover we take $h(\lambda)=(\psi(\lambda))^\frac{2-2\delta}{1-2\delta} = \lambda^{-\theta}$, for $\theta=\frac{2-2\delta}{1-2\delta}\alpha\in(0,1)$, provided that $\delta$ is sufficiently small. Finally 
Lemma \ref{lem3} implies that assumption (\ref{kpzal}) is also satisfied. Therefore, for ${\mathcal R}$ defined by (\ref{r}), there exists a constant $C_2>0$ (note that $C_2$ depends on $\alpha,\delta$ and $C_1$), such that
\begin{eqnarray}\label{dwa}
II_{ijk}^\epsilon&\leq& \eta C\left(C_2\int_{T_k}^{T_{k+1}-\epsilon}|\dot{v}_j-\dot{v}_i|dt + {\mathcal R}(x_j-x_i,T_{k+1}-\epsilon)-{\mathcal R}(x_j-x_i,T_{k})\right)\nonumber\\
&\leq&  \eta C\left(C_2\int_{T_k}^{T_{k+1}-\epsilon}|\dot{v}_j|dt
+ C_2\int_{T_k}^{T_{k+1}-\epsilon}|\dot{v}_i|dt\right.\nonumber\\ &+& \left. {\mathcal R}(x_j-x_i,T_{k+1}-\epsilon)-{\mathcal R}(x_j-x_i,T_{k})\right)
\end{eqnarray}
and by combining (\ref{jeden}) with (\ref{dwa}) we end up with the estimation
\begin{eqnarray*}
\frac{1}{N}\sum_{i=1}^N\int_{T_k}^{T_{k+1}-\epsilon}|\dot{v}_i|dt&\leq& \frac{1}{N^2}\sum_{i,j=1}^NC(\eta)\left(\int_{T_k}^{T_{k+1}-\epsilon}|v_j-v_i|^2\psi(|x_j-x_i|)dt\right)^\frac{2\delta}{1+2\delta} \cdot(T_{k+1}-\epsilon-T_{k})^\frac{1}{1+2\delta} \\ &+&2\eta CC_2\frac{1}{N}\sum_{i=1}^N\int_{T_k}^{T_{k+1}-\epsilon}|\dot{v}_i|dt + \frac{\eta C}{N^2}\sum_{i,j=1}^N\left({\mathcal R}(x_j-x_i,T_{k+1}-\epsilon)-{\mathcal R}(x_j-x_i,T_{k})\right),
\end{eqnarray*}
which assuming that $\eta=\frac{1}{4 C C_2}$ leads to
\begin{eqnarray*}
\frac{1}{N}\sum_{i=1}^N\int_{T_k}^{T_{k+1}-\epsilon}|\dot{v}_i|dt &\leq& \frac{2C}{N^2}\sum_{i,j=1}^N\left(\int_{T_k}^{T_{k+1}-\epsilon}|v_j-v_i|^2\psi(|x_j-x_i|)dt\right)^\frac{2\delta}{1+2\delta} \cdot(T_{k+1}-\epsilon-T_{k})^\frac{1}{1+2\delta} \\ &+&\frac{1}{2N^2 CC_2}\sum_{i,j=1}^N\left({\mathcal R}(x_j-x_i,T_{k+1}-\epsilon)-{\mathcal R}(x_j-x_i,T_{k})\right).
\end{eqnarray*}
By the monotone convergence theorem and continuity of ${\mathcal R}$ (see the end of the proof of Lemma \ref{lemkp}), we may pass with $\epsilon$ to $0$ obtaining
\begin{eqnarray*}
\frac{1}{N}\sum_{i=1}^N\int_{T_k}^{T_{k+1}}|\dot{v}_i|dt&\leq& \frac{2C}{N^2}\sum_{i,j=1}^N\left(\int_{T_k}^{T_{k+1}}|v_j-v_i|^2\psi(|x_j-x_i|)dt\right)^\frac{2\delta}{1+2\delta} \cdot(T_{k+1}-T_{k})^\frac{1}{1+2\delta}\\
&+&\frac{1}{2N^2CC_2}\sum_{i,j=1}^N\left({\mathcal R}(x_j-x_i,T_{k+1})-{\mathcal R}(x_j-x_i,T_{k})\right)
\end{eqnarray*}
and finally sum over $k=0,...,N_s$ to get
\begin{eqnarray}\label{regpom1}
\frac{1}{N}\sum_{i=1}^N\int_0^{T}|\dot{v}_i|dt&\leq& \frac{2C}{N^2}\sum_{i,j=1}^N\sum_{k=1}^{N_s}\left(\int_{T_k}^{T_{k+1}}|v_j-v_i|^2\psi(|x_j-x_i|)dt\right)^\frac{2\delta}{1+2\delta} \cdot(T_{k+1}-T_{k})^\frac{1}{1+2\delta}\nonumber\\
&+&\frac{1}{2N^2CC_2}\sum_{i,j=1}^N\left({\mathcal R}(x_j-x_i,T)-{\mathcal R}(x_j-x_i,0)\right) =: I+II.
\end{eqnarray}
We yet again apply H\" older's inequality (this time for sums) with exponents $q=\frac{1+2\delta}{2\delta}$ and $q^{'}=1+2\delta$ along with Lemma \ref{lem1} to get
\begin{eqnarray}\label{regpom2}
I\leq\frac{2C}{N^2}\sum_{i,j=1}^N\left(\int_0^{T}|v_j-v_i|^2\psi(|x_j-x_i|)dt\right)^\frac{2\delta}{1+2\delta} \cdot T^\frac{1}{1+2\delta}\leq 2CC_1^2\cdot T^\frac{1}{1+2\delta}.
\end{eqnarray}
Moreover by the definition of ${\mathcal R}$ and Proposition \ref{sum}.3, we have
\begin{eqnarray}\label{regpom3}
II&\leq&\frac{1}{2N^2CC_2}\sum_{i,j=1}^N\left(|v_j(T)-v_i(T)||H(|x_j(T)-x_i(T)|)|+|v_j(0)-v_i(0)||H(|x_j(0)-x_i(0)|)|\right)\nonumber\\
&\leq&\frac{2}{(1-\alpha)CC_2}C_1^{2-\alpha}T^{1-\alpha}.
\end{eqnarray}

After combining inequalities (\ref{regpom1}), (\ref{regpom2}) and (\ref{regpom3}), we obtain
\begin{eqnarray*}
\frac{1}{N}\sum_{i=1}^N\int_0^T|\dot{v_i}|dt\leq 2CC_1^2\cdot T^\frac{1}{1+2\delta} + \frac{2}{(1-\alpha)C_2}C_1^{2-\alpha}T^{1-\alpha}=: M,
\end{eqnarray*}
which finishes the proof.
\end{proof}

\subsection{Proofs of Lemmas \ref{lem1} and \ref{lem3}}
In this section we present the technical yet straightforward proofs of Lemmas \ref{lem1} and \ref{lem3}.
\begin{proof}[Proof of Lemma \ref{lem1}]
We have
\begin{eqnarray*}
\int_0^TR(t)dt = \sum_{k=0}^{N_s}\int_{T_k}^{T_{k+1}}R(t)dt,
\end{eqnarray*}
with $T_k$ and $N_s$ from Definition \ref{sol}. By Proposition \ref{sum}.1, the function 
\begin{eqnarray*}
r(t):=\sum_{i,j=1}^N(v_i(t)-v_j(t))^2
\end{eqnarray*}
is absolutely continuous on each interval $[T_k,T_{k+1}-\epsilon]$ with arbitrarily small $\epsilon>0$. Then, by $(\ref{cs})_2$ on each such interval we have
\begin{eqnarray*}
\frac{d}{dt}r&=&2\sum_{i,j=1}^N(v_i-v_j)\left(\frac{1}{N}\sum_{k=1}^N(v_k-v_i)\psi(|x_i-x_k|)- \frac{1}{N}\sum_{k=1}^N(v_k-v_j)\psi(|x_j-x_k|)\right)\\
&=& \frac{2}{N}\sum_{i,j,k=1}^N(v_i-v_j)(v_k-v_i)\psi(|x_i-x_k|)- \frac{2}{N}\sum_{i,j,k=1}^N(v_i-v_j)(v_k-v_j)\psi(|x_j-x_k|).
\end{eqnarray*}
We substitute $i$ and $k$ in the first summand and $j$ and $k$ in the second summand to obtain
\begin{eqnarray*}
\frac{d}{dt}r&=&\frac{1}{N}\sum_{i,j,k=1}^N(v_i-v_j)(v_k-v_i)\psi(|x_i-x_k|) +\frac{1}{N}\sum_{i,j,k=1}^N(v_k-v_j)(v_i-v_k)\psi(|x_i-x_k|)\\
&-&\frac{1}{N}\sum_{i,j,k=1}^N(v_i-v_j)(v_k-v_j)\psi(|x_j-x_k|) -\frac{1}{N}\sum_{i,j,k=1}^N(v_i-v_k)(v_j-v_k)\psi(|x_j-x_k|)\\
&=&-\frac{1}{N}\sum_{i,j,k=1}^N(v_i-v_k)^2\psi(|x_i-x_k|) -\frac{1}{N}\sum_{i,j,k=1}^N(v_j-v_k)^2\psi(|x_j-x_k|)\\
&=&-2\sum_{i,j=1}^N(v_i-v_j)^2\psi(|x_i-x_j|) = -2R.
\end{eqnarray*}
Therefore
\begin{eqnarray*}
\int_{T_k}^{T_{k+1}-\epsilon} Rdt = \frac{1}{2}\left(r(T_k)-r(T_{k+1}-\epsilon)\right) 
\end{eqnarray*}
and thus, by the monotone convergence theorem and continuity of $r$, we pass to the limit with $\epsilon\to 0$ obtaining
\begin{eqnarray}\label{lem11}
\int_{T_k}^{T_{k+1}} Rdt =  \frac{1}{2}\left(r(T_k)-r(T_{k+1})\right).
\end{eqnarray}
Finally, we take a sum over all $k=0,...,N_s$ of the equations of the form (\ref{lem11}) to get
\begin{eqnarray*}
\int_0^TR(t)dt = \frac{1}{2}\left(r(0)-r(T)\right)\leq C_1N^2,
\end{eqnarray*}
where the final estimation is justified by Proposition \ref{sum}.3.
\end{proof}
\begin{proof}[Proof of Lemma \ref{lem3}]
Given $i,j=1,...,N$, we have
\begin{eqnarray}\label{l3pom}
\int_{s_1}^{s_2}|x_j-x_i|^{-\theta}dt = \sum_k\int_{t_{k-1}}^{t_k}|x_j-x_i|^{-\theta}dt,
\end{eqnarray}
where $t_k$ denote the times of collision of $x_j$ and $x_i$ that happen in the time interval $[{s_1},{s_2}]$. By Proposition \ref{sum}.3, the only density points of the times of collision are times of sticking and since there are no times of sticking in $[{s_1},{s_2}]$ -- the sum on the right-hand side of (\ref{l3pom}) is finite. Thus it is sufficient to show that each summand is finite (even if it is arbitrarily large), hence from this point we fix $k$. Now, if the particles do not stick together in $[{s_1},{s_2}]$, then for $t\in[{s_1},{s_2}]$ either $x_i(t)\neq x_j(t)$ or $v_i(t)\neq v_j(t)$. In particular $v_j(t_{k-1})-v_i(t_{k-1}) =: v_{k-1}\neq 0$ and $v_j(t_{k})-v_i(t_{k}) =: v_k\neq 0$ and by continuity of $v$ (see Proposition \ref{sum}), there exist positive $\rho$ and $\delta$, such that
\begin{eqnarray*}
v_j-v_i\in B(v_{k-1},\rho)&\ \ \ {\rm in}&\ \ \ [t_{k-1},t_{k-1}+\delta]\ \ \ {\rm and}\\
v_j-v_i\in B(v_k,\rho)&\ \ \ {\rm in}&\ \ \ [t_k-\delta,t_k]
\end{eqnarray*}
and $0$ does not belong to neither $B(v_{k-1},\rho)$ nor $B(v_k,\rho)$. Let us split the integral from the right-hand side of (\ref{l3pom}) in the following manner:
\begin{eqnarray*}
\int_{t_{k-1}}^{t_k}|x_j-x_i|^{-\theta}dt = \left(\int_{t_{k-1}}^{t_{k-1}+\delta}+\int_{t_{k-1}+\delta}^{t_k-\delta}+\int_{t_{k}-\delta}^{t_k}\right) |x_j-x_i|^{-\theta}dt =: I + II + III.
\end{eqnarray*}
Then there exists an arbitrarily large constant $C(\delta)$, that bounds $II$ from the above since $|x_j-x_i|$ is continuous and nonzero on $[t_{k-1}+\delta, t_k-\delta]$. To estimate $I$ we notice that for $t\in[t_{k-1}, t_{k-1}+\delta]$ it holds:
\begin{eqnarray*}
|x_j(t)-x_i(t)| = \left|\int_{t_{k-1}}^tv_j-v_i\right|\geq \inf_{\xi\in B(v_{k-1},\rho)}|\xi|(t-t_{k-1})\geq c(t-t_{k-1})
\end{eqnarray*}
for some small constant $c>0$. Thus
\begin{eqnarray*}
\int_{t_{k-1}}^{t_{k-1}+\delta}|x_j-x_i|^{-\theta}dt\leq c^{-\theta}\int_{t_{k-1}}^{t_{k-1}+\delta}(t-t_{k-1})^{-\theta}dt<\infty,
\end{eqnarray*}
since $\theta<1$. Estimation of $III$ proceeds simiralry to the estimation of $I$.
\end{proof}
\subsection{Uniqueness of solutions}\label{uni}
Our goal in this section is to prove uniqueness of solutions to (\ref{cs}) for $\alpha\in(0,\frac{1}{2})$.
\begin{prop}\label{un1}
With the assumptions of Theorem \ref{main}, the $W^{2,1}$ weak solution of (\ref{cs}) is unique.
\end{prop}
\begin{proof}
Suppose that $(x^1,v^1)$ and $(x^2,v^2)$ are two $W^{2,1}$ weak solutions of $(\ref{cs})$, with weight $\psi$ given by (\ref{psi}) and $\alpha\in(0,\frac{1}{2})$ on the time interval $[0,T]$, subjected to the initial data $(x_0,v_0)$. We will show that in fact $(x^1,v^1)\equiv (x^2,v^2)$. The proof will follow by four steps. In steps 1-3 we prove uniqueness in a small neighborhood of the initial time $t=0$ considering three cases: non-collision initial data, non-sticking initial data and initial data with particles that are stuck together. In step 4 we combine our efforts from previous steps and conclude the proof.
\begin{description}
\item[Step 1.]
If there are no collisions at the initial time, which means that for all $i\neq j$, we have $x_{0,i}\neq x_{0,j}$, then by the fact that $x^1,x^2\in C^1([0,T])$, there exists $\delta>0$, such that for all $i\neq j$, we have $|x^m_i(s)-x^m_j(s)|>\delta$ with $m=1,2$ for $s\in[0,\delta]$. The communication weight $\psi$ is smooth on the domain $[\delta,+\infty)$ and thus, on the time interval $[0,\delta]$ system $(\ref{cs})$ is a nonlinear ODE with a Lipschitz continuous nonlinearity and uniqueness is standard.\\
\item[Step 2.]
In the case of no sticking at the initial time (which means that for all $i\neq j$ if $x_{0,i}=x_{0,j}$ then $v_{0,i}\neq v_{0,j}$) let us consider
\begin{eqnarray*}
r(t):=\sum_{i=1}^N(v_i^1(t)-v_i^2(t))^2.
\end{eqnarray*}
By the assumptions $r$ is an absolutely continuous function, thus it has a bounded variation and can be represented as a sum of two functions, respectively nonincreasing and nondecreasing. Noting that $r(0)=0$, let
\begin{eqnarray*}
r_{inc}(t) := \int_0^t(\dot{r}(s))_+ds,
\end{eqnarray*}
where by $(\dot{r})_+$ we denote the positive part of the function $\dot{r}$. Then if we prove that $r_{inc} \equiv 0$ then we will also know that $r\equiv 0$ and that actually $x^1\equiv x^2$.
By $(\ref{cs})_2$, we have
\begin{eqnarray*}
\frac{d}{dt}r_{inc} =\frac{2}{N}\left[\sum_{i,j=1}^N(v_i^1-v_i^2)\left((v_j^1-v_i^1)\psi(|x^1_j-x^1_i|)- (v_j^2-v_i^2)\psi(|x^2_j-x^2_i|)\right)\right]_+.
\end{eqnarray*}
After substituting $i$ and $j$ in the above equation we obtain
\begin{eqnarray}\label{st20}
\frac{d}{dt}r_{inc} &=& \frac{1}{N}\left[\sum_{i,j=1}^N\left((v_i^1-v_i^2)-(v_j^1-v_j^2)\right) \left((v_j^1-v_i^1)\psi(|x^1_j-x^1_i|)- (v_j^2-v_i^2)\psi(|x^2_j-x^2_i|)\right)\right]_+\nonumber\\
 &=& \frac{1}{N}\left[\sum_{i,j=1}^N\left((v_i^1-v_i^2)-(v_j^1-v_j^2)\right) \left((v_j^1-v_i^1)-(v_j^2-v_i^2)\right)\psi(|x_j^1-x_i^1|)\right.\nonumber\\ &+& \left.\sum_{i,j=1}^N\left((v_i^1-v_i^2)-(v_j^1-v_j^2)\right)(v_j^2-v_i^2)\left(\psi(|x_j^1-x_i^1|)-\psi(|x_j^2-x_i^2|)\right)\right]_+\nonumber\\
&=&\frac{1}{N}\left[-\sum_{i,j=1}^N\left((v_i^1-v_i^2)-(v_j^1-v_j^2)\right)^2\psi(|x_j^1-x_i^1|)\right.\nonumber\\
 &+& \left.\sum_{i,j=1}^N\left((v_i^1-v_i^2)-(v_j^1-v_j^2)\right)(v_j^2-v_i^2)\left(\psi(|x_j^1-x_i^1|)-\psi(|x_j^2-x_i^2|)\right)\right]_+\nonumber\\
&\leq&\frac{1}{N}\sum_{i,j=1}^N\left|(v_i^1-v_i^2)-(v_j^1-v_j^2)\right||v_j^2-v_i^2|\left|\psi(|x_j^1-x_i^1|)-\psi(|x_j^2-x_i^2|)\right|.
\end{eqnarray}
By Proposition \ref{sum}.3 the factor $|v_j^2-v_i^2|$ is bounded uniformly with respect to $i,j$ and $t$. Next, we fix $i$ and $j$ and consider two cases:
\begin{description}
\item[Case 1: $x_i(0)\neq x_j(0)$.]

This is in fact the situation from step 1, i.e. there exists $\delta>0$, such that for all $i,j$ with $x_i(0)\neq x_j(0)$, we have
\begin{eqnarray*}
|x_i^m-x_j^m|\geq\delta,\ \ \ \ \ \ m=1,2
\end{eqnarray*}
on $[0,\delta]$. Then
\begin{eqnarray}\label{st21}
\left|\psi(|x_j^1(t)-x_i^1(t)|)-\psi(|x_j^2(t)-x_i^2(t)|)\right|\leq L(\delta)\left|(x_j^1(t)-x_i^1(t))-(x_j^2(t)-x_i^2(t))\right|
\end{eqnarray}
for some Lipschitz constant $L(\delta)$.
\item[Case 2: $x_i(0) = x_j(0)$.]
Let us recall that in this step we assume that if $x_i(0) = x_j(0)$ then $v_i(0) \neq v_j(0)$. Therefore for our $i$ and $j$ we have $v_j(0)-v_i(0)=:v_{ji}\neq 0$ and by continuity of $v^1$ and $v^2$ there exist $\delta>0$ such that
\begin{eqnarray*}
|v_i^m-v_j^m|\geq \delta,\ \ \ \ \ \ m=1,2,
\end{eqnarray*}
which implies that
\begin{eqnarray*}
|x_i^m(s)-x_j^m(s)|\geq \frac{1}{2}\delta s
\end{eqnarray*}
on $[0,\delta]$ for all $i,j$ and $m=1,2$. Thus by mean value theorem
\begin{eqnarray}\label{st22}
\left|\psi(|x_j^1-x_i^1|)-\psi(|x_j^2-x_i^2|)\right| &\leq& C\left|(x_j^1-x_i^1)-(x_j^2-x_i^2)\right| \int_0^1\left|\theta|x_j^1-x_i^1|+(1-\theta)|x_j^2-x_i^2|\right|^{-1-\alpha}d\theta\nonumber\\
&\leq& C\left|(x_j^1-x_i^1)-(x_j^2-x_i^2)\right|\left|\frac{\delta}{2} t\right|^{-1-\alpha}\nonumber\\
&\leq& C(\delta)\frac{\left|(x_j^1-x_i^1)-(x_j^2-x_i^2)\right|}{t}\left|t\right|^{-\alpha}.
\end{eqnarray}
\end{description} 
Moreover in either Case 1 or Case 2
\begin{eqnarray}\label{st23}
\left|(x_j^1(t)-x_i^1(t))-(x_j^2(t)-x_i^2(t))\right|&\leq& t\sup_{s\in[0,t]}\left|(v_j^1(s)-v_i^1(s))-(v_j^2(s)-v^2_i(s))\right|\nonumber\\
&\leq& 2t\sup_{s\in[0,t]}\sqrt{r(s)}
\leq 2t\sup_{s\in[0,t]}\sqrt{r_{inc}(s)}\leq 2t\sqrt{r_{inc}(t)}
\end{eqnarray}
and thus by combining inequalities (\ref{st20}), (\ref{st21}), (\ref{st22}) and (\ref{st23}) with H\" older's inequality one obtains
\begin{eqnarray*}
\frac{d}{dt}r_{inc}\leq Cr_{inc}\cdot f,
\end{eqnarray*}
where
\begin{eqnarray*}
f(t):=\max\{2L(\delta)t,2C(\delta)|t|^{-\alpha}\},
\end{eqnarray*}
which is an integrable function. Therefore Gronwall's lemma implies that the solution is unique at least on $[0,\delta]$ for a sufficiently small, positive $\delta$.\\
\item[Step 3.] The purpose of this step is to prove uniqueness in case, when at least two particles are stuck together at the initial time, i.e. $x_{0,i}=x_{0,j}$ and $v_{0,i}=v_{0,j}$ for some $i,j = 1,...,N$. We present this step as a consequence of the presented below lemma.
\begin{lem}\label{stuck}
Suppose that at some time $t_0\in[0,T]$ and some $i,j=1,...,N$, we have $x_i(t_0)=x_j(t_0)$ and $v_i(t_0)=v_j(t_0)$. Then $x_i\equiv x_j$ on $[t_0,t_0+\delta]$ for some positive $\delta$.
\end{lem}
The above lemma in particular implies that on $[0,\delta]$ any particles that are stuck together can be treated as a single particle. From the point of view of uniqueness it means that we do not have to consider the case, when two or more particles are stuck together, since they cannot separate anyway. Thus if only the trajectory on which they move is unique then their respective trajectories are unique too (since in fact they are the same).
\begin{proof}[Proof of Lemma \ref{stuck}]
The proof follows similarly to that of step 2. Let
\begin{eqnarray*}
r(t):=\sum_{i,j\in[i]}(v_i(t)-v_j(t))^2,
\end{eqnarray*}
where $[i]$ denotes the set of those $j$ that $x_j(t_0)=x_i(t_0)$ and $v_j(t_0)=v_i(t_0)$. Therefore if we show that $r\equiv 0$ then the thesis of Lemma \ref{stuck} will be satisfied. We have
\begin{eqnarray}
\frac{d}{dt}r_{inc} &=& \frac{2}{N}\left[\sum_{i,j\in[i]}\sum_{k=1}^N(v_i-v_j) \left((v_k-v_i)\psi(|x_k-x_i|)-(v_k-v_j)\psi(|x_k-x_j|)\right)\right]_+\nonumber\\
&\leq& \frac{2}{N}\left(\sum_{i,j,k\in[i]}+\sum_{i,j\in[i]}\sum_{k\notin[i]}\right)\left[(v_i-v_j) \left((v_k-v_i)\psi(|x_k-x_i|)-(v_k-v_j)\psi(|x_k-x_j|)\right)\right]_+\nonumber\\
&\stackrel{see\ below}{\leq}&\frac{2}{N}\sum_{i,j\in[i]}\sum_{k\notin[i]}\left[(v_i-v_j) \left((v_k-v_i)\psi(|x_k-x_i|)-(v_k-v_j)\psi(|x_k-x_j|)\right)\right]_+\label{st31}\\
&\leq&\frac{2}{N}\sum_{i,j\in[i]}\sum_{k\notin[i]}\left[-(v_i-v_j)^2\psi(|x_k-x_i|)\right]_+ +\left[(v_i-v_j)(v_k-v_j)(\psi(|x_k-x_i|)-\psi(|x_k-x_j|))\right]_+\nonumber\\
&\leq& \frac{C}{N}\sum_{i,j\in[i]}\sum_{k\notin[i]}|v_i-v_j||\psi(|x_k-x_i|)-\psi(|x_k-x_j|)|.\nonumber
\end{eqnarray}
Inequality (\ref{st31}) follows by the fact that in the triple sum over the set $[i]$ the indexes may be substituted in the same fashion as in the proof of Lemma \ref{lem1}.
We estimate $|\psi(|x_k-x_i|)-\psi(|x_k-x_j|)|$ similarly to estimations from Case 1 and Case 2 in the previous step obtaining altogether
\begin{eqnarray*}
\frac{d}{dt}r_{inc}\leq C r_{inc}\cdot f
\end{eqnarray*}
for some integrable function $f$. Then by Gronwall's lemma $r_{inc}\equiv 0$ on $[0,\delta]$, which means that also $r\equiv 0$ on $[0,\delta]$ and that for all $i,j\in[i]$ we have $x_i\equiv x_j$ on $[0,\delta]$.
\end{proof}
\item[Step 4.] In this step we finish the proof of uniqueness by putting together all the information obtained in previous steps. Suppose that we have two distinct solutions $(x^1,v^1)$ and $(x^2,v^2)$ originating in $(x_0,v_0)$. Then, regardless of the initial data, by all three previous steps, there exists an interval $[0,\delta]$ on which $x^1\equiv x^2=: x$. Without a loss of generality we may assume that for $t=\delta$ we have $x_i(t)\neq x_j(t)$ or $x_i\equiv x_j$ on $[0,\delta]$ for all $i,j=1,...,N$. Therefore, by step 1 and step 3 we may prolong the interval on which $x^1\equiv x^2$. In fact we may prolong it as long as there is no collision between any particles. Let $t_0$ be the first time of collision. Then by step 1 and step 3, the uniqueness is ensured up to $t_0-\epsilon$ for arbitrarily small $\epsilon>0$. Now, by Proposition \ref{sum}, $(x,v)$ is continuous on whole $[0,T]$, thus it has a unique left sided limit at $t_0$, which prolongs uniqueness up to $t_0$. Finally we may treat $t_0$ as the new starting point and obtain uniqueness on $[t_0,t_1]$. Therefore the solution is unique between any two times of collision and the (possibly infinite) sum of such intervals include all $[0,T]$.
\end{description}  
\end{proof}
We end this section with the proof of uniqueness of piecewise weak solutions.
\begin{proof}[Proof of Theorem \ref{main2}]
The proof is almost exactly the same as of Proposition \ref{un1}. The first difference is that the function $r$ from step 2 was absolutely continuous by the fact that the solutions were $W^{2,1}$ weak on $[0,T]$, while this time they are $W^{2,1}$ weak on each interval $[T_k,T_{k+1}-\epsilon]$ as stated in Proposition \ref{sum}.1. This however is of no difference since we need $r$ to be absolutely continuous only on $[0,\delta]$ for some small $\delta>0$. The second difference is that this time we actually do not need Lemma \ref{stuck} since by Definition \ref{sol} and in particular by the use of sets $B_i(t)$ (defined in (\ref{bi})) we already ensured that the trajectories remain stuck together indefinitely.
\end{proof}
{\bf Acknowledgements.} I would like to thank Piotr B. Mucha for helpful remarks and inspirational discussions. This work was supported by International Ph.D. Projects Programme of Foundation for Polish Science operated within the Innovative Economy Operational Programme 2007-2013 funded by European Regional Development Fund (Ph.D. Programme:
Mathematical Methods in Natural Sciences) and partially supported by the Polish NCN grant PRELUDIUM no. 2013/09/N/ST1/04113.
\section{Appendix}\label{app}
In the appendix we present interpolation inequality which was crucial in the proof of Proposition \ref{abs}. We moved it here from section \ref{reg}, since the proof is self-contained and does not refer directly to the subject of the rest of the paper. This inequality along with it's proof comes in almost unchanged form from paper \cite{ka1} but we present the proof anyway for readers convenience.
\begin{lem}\label{lemkp}
Let $f=(f_1,...,f_d):[0,T]\to \r^d$ be a $C^1([0,T])\cap W^{2,1}_{loc}((0,T))$ vector valued function that is nonzero a.e.. Moreover let $h:[0,\infty)\to[0,\infty)$ be defined as
\begin{eqnarray*}
h(\lambda)=\lambda^{-\theta},
\end{eqnarray*}
for some $0<\theta <1$. Then there exists a constant $C_2>0$ depending on $\|f\|_\infty$ and $\theta$, such that we have
\begin{eqnarray}\label{kp}
\int_0^T|f^{'}|^2h(|f|)dt\leq C_2\int_0^T|f^{''}|dt + {\mathcal R}(f,T)-{\mathcal R}(f,0),
\end{eqnarray}
provided that
\begin{eqnarray}\label{kpzal}
\int_\epsilon^{T-\epsilon}h(|f|)dt <\infty
\end{eqnarray}
for all $\epsilon>0$. Here, for $H(\lambda) = \frac{1}{1-\theta}\lambda^{1-\theta}$ -- a primitive of $h$, we denote
\begin{eqnarray}\label{r}
{\mathcal R}(f,t):=
\left\{
\begin{array}{ccc}
\frac{f(t)f^{'}(t)}{|f(t)|}H(|f(t)|)& {\rm for}& f(t)\neq 0,\\
0& {\rm for}& f(t)=0.
\end{array}
\right.
\end{eqnarray}
\end{lem}
\begin{proof}
We assume that
\begin{eqnarray}\label{kpzal2}
\int_0^T|f^{''}|dt<\infty,
\end{eqnarray}
since otherwise, inequality (\ref{kp}) holds with infinity on the right-hand side.
For $\eta\geq 0$, let us define
\begin{eqnarray*}
f_\eta:=\sqrt{f^2+\eta}.
\end{eqnarray*}
Then $f_\eta$ is a bounded (uniformly for $0<\eta<1$) real function, such that
\begin{eqnarray}\label{kp3}
\max\{|f|,\sqrt{\eta}\}\leq f_\eta\leq \|f\|_\infty +1.
\end{eqnarray}
Moreover an easy computation shows that
\begin{eqnarray*}
f_\eta^{'}&=&\frac{f\cdot f^{'}}{f_\eta},\\
f_\eta^{''}&=&\frac{f\cdot f^{''}}{f_\eta} + \frac{(f^{'})^2}{f_\eta}\cdot\frac{\eta}{f^2+\eta}.
\end{eqnarray*}
First, let us prove an assertion for $f_\eta$ that is similar to (\ref{kp}). Namely we aim to show that given $\epsilon>0$, such that $f(\epsilon)\neq 0$ and $f(T-\epsilon)\neq 0$, we have
\begin{eqnarray}\label{kp2}
{\mathcal L}_\eta := \int_\epsilon^{T-\epsilon}|f_\eta^{'}|^2 h(|f_\eta|)dt &\leq& \int_\epsilon^{T-\epsilon}|f_\eta^{''}||H(f_\eta)|dt + \left[f_\eta^{'}(T-\epsilon)H(f_\eta(T-\epsilon)) - f_\eta^{'}(\epsilon)H(f_\eta(\epsilon))\right]\nonumber\\ &=:&{\mathcal R1}_\eta + {\mathcal R2}_\eta.
\end{eqnarray}
Since
\begin{eqnarray*}
{\mathcal L}_\eta = \int_\epsilon^{T-\epsilon}f_\eta^{'}\cdot f_\eta^{'}h(f_\eta)dt,
\end{eqnarray*}
after integrating the right-hand side in the above equation by parts (which is justified by the fact that $f_\eta\in W^{2,1}_{loc}((0,T))$), we obtain
\begin{eqnarray*}
{\mathcal L}_\eta &=& -\int_\epsilon^{T-\epsilon}f_\eta^{''}\cdot H(f_\eta) dt + f_\eta^{'}(T-\epsilon)H(f_\eta(T-\epsilon)) - f_\eta^{'}(\epsilon)H(f_\eta(\epsilon))\\
&\leq& \int_\epsilon^{T-\epsilon}|f_\eta^{''}||H(f_\eta)|dt + f_\eta^{'}(T-\epsilon)H(f_\eta(T-\epsilon)) - f_\eta^{'}(\epsilon)H(f_\eta(\epsilon)),
\end{eqnarray*}
which finishes the proof of (\ref{kp2}). As expected, our next step is to converge with $\eta\to 0$ and obtain (\ref{kp}).
First let us deal with ${\mathcal R1}_\eta$. We have a.e.
\begin{eqnarray}
f_\eta&\searrow&|f|,\nonumber\\
|f_\eta^{''}|&\to& |f^{''}|,\nonumber\\
|f_\eta^{''}|&\leq& |f^{''}|+\frac{(f^{'})^2}{f_\eta},\label{star}
\end{eqnarray}
which implies that the integrand appearing in ${\mathcal R1}_\eta$ converges a.e. to $|f^{''}||H(|f|)|$. To ensure convergence of the integrals we will apply Lebesgue's dominated convergence theorem. To do this let us note, that
\begin{eqnarray*}
|H(f_\eta)|\leq \frac{1}{1-\theta}(\|f_\eta\|_\infty)^{1-\theta}
\end{eqnarray*}
and by (\ref{kp3}) and (\ref{star}) we have
\begin{eqnarray}\label{kppom}
|f_\eta^{''}||H(f_\eta)|\leq |f^{''}||H(f_\eta)| + \frac{|f^{'}|^2}{f_\eta}|H(|f_\eta|)|\leq C|f^{''}| + \frac{1}{1-\theta}|f^{'}|^2h(|f|),
\end{eqnarray}
where $C=\frac{1}{1-\theta}(\|f\|_\infty+1)^{1-\theta}$. However by (\ref{kpzal}) and (\ref{kpzal2}), the right-hand side in (\ref{kppom}) is integrable on $[\epsilon, T-\epsilon]$. Therefore by Lebesgue's dominated convergence theorem
\begin{eqnarray*}
\int_\epsilon^{T-\epsilon}|f_\eta^{''}||H(f_\eta)|dt\to\int_\epsilon^{T-\epsilon}|f^{''}|H(|f|)|dt\leq C\int_\epsilon^{T-\epsilon}|f^{''}|dt.
\end{eqnarray*}
Next we converge with ${\mathcal R2}_\eta$. We note that
\begin{eqnarray*}
f_\eta^{'}(t)\to \frac{f(t)\cdot f^{'}(t)}{|f(t)|},
\end{eqnarray*}
as long as $f(t)\neq 0$, which by the choice of $\epsilon$ is the case for $t=\epsilon$ and $t=T-\epsilon$. Thus
\begin{eqnarray*}
{\mathcal R2}_\eta\to \left[\frac{f(T-\epsilon)f^{'}(T-\epsilon)}{|f(T-\epsilon)|}H(|f(T-\epsilon)|) - \frac{f(\epsilon)f^{'}(\epsilon)}{|f(\epsilon)|}H(|f(\epsilon)|)\right].
\end{eqnarray*}
Lastly, by (\ref{kp2}) and Fatou's lemma
\begin{eqnarray*}
\int_\epsilon^{T-\epsilon}|f^{''}|^2h(|f|)dt&\leq& \liminf_{\eta\to 0} {\mathcal R1}_\eta+{\mathcal R2}_\eta\\
&\leq& C\int_\epsilon^{T-\epsilon}|f^{''}|dt + \left[\frac{f(T-\epsilon)f^{'}(T-\epsilon)}{|f(T-\epsilon)|}H(|f(T-\epsilon)|) - \frac{f(\epsilon)f^{'}(\epsilon)}{|f(\epsilon)|}H(|f(\epsilon)|)\right].
\end{eqnarray*}
The final step of the proof is to converge with $\epsilon\to 0$. This, however is straightforward by the monotone convergence theorem and by the fact that $f\in C^1([0,T])$. The only non-trivial part is that for the sake of convenience we choose a suitable sequence $\epsilon_n\searrow 0$, such that for all $n$, we have $f(\epsilon)\neq 0\neq f(T-\epsilon)$, which we can do since $f\neq 0$ a.e. in $[0,T]$. It is also worthwhile to note that we use the fact that the function
\begin{eqnarray*}
t\mapsto \frac{f(t)f^{'}(t)}{|f(t)|}H(|f(t)|) = \frac{f(t)}{|f(t)|^{\theta}}f^{'}(t)
\end{eqnarray*} 
is continuous, since $f, f^{'}$ and the function $\lambda\mapsto\frac{\lambda}{|\lambda|^{\theta}}$ are continuous (continuity at $0$ follows from the assumption that $\theta<1$).
\end{proof}
\begin{rem}\rm
Similar equation with multiple examples and applications can be found in \cite{ka1} or \cite{ka2}.
\end{rem}
\bibliographystyle{abbrv}
\bibliography{2}
\end{document}